\documentclass[11pt, a4paper, USenglish]{amsart}

\usepackage[USenglish]{babel}
\usepackage{mathtools}
\usepackage{graphicx}
\usepackage{amsmath,amsthm,amssymb}
\usepackage{geometry}
\usepackage{tikz}
\usepackage{tikz-cd}
\usepackage[hidelinks]{hyperref}
\usepackage{csquotes}
\usepackage{longtable}
\usepackage{enumitem}
\usepackage{caption}
\usepackage{quiver}
\usepackage[
backend=biber,
style=alphabetic,
sorting=nyt,
maxbibnames=100,
isbn=false,
doi=false,
url=false
]{biblatex}

\newcommand{\calX}{\mathcal{X}}
\newcommand{\calY}{\mathcal{Y}}
\newcommand{\FFp}{\mathbb{F}_p}
\newcommand{\bFFp}{\overline{\mathbb{F}}_p}
\newcommand{\FFpn}{\mathbb{F}_{p^n}}

\DeclareMathOperator{\Pic}{\mathrm{Pic}}
\DeclareMathOperator{\Aut}{\mathrm{Aut}}
\DeclareMathOperator{\Spec}{\mathrm{Spec}}
\DeclareMathOperator{\Spf}{\mathrm{Spf}}
\DeclareMathOperator{\NS}{\mathrm{NS}}
\DeclareMathOperator{\Def}{\mathrm{Def}}
\DeclareMathOperator{\rk}{\mathrm{rk}}

\theoremstyle{plain}
\newtheorem{theorem}{Theorem}[section]
\newtheorem{proposition}[theorem]{Proposition}
\newtheorem{lemma}[theorem]{Lemma}
\newtheorem{cor}[theorem]{Corollary}

\theoremstyle{definition}
\newtheorem{definition}[theorem]{Definition}
\newtheorem*{notation}{Notation}
\newtheorem{remark}[theorem]{Remark}

\newtheorem*{acknowledgments}{Acknowledgments}
\newtheorem*{setup}{Setup}

\newcommand{\nocontentsline}[3]{}
\let\origcontentsline\addcontentsline
\newcommand\stoptoc{\let\addcontentsline\nocontentsline}
\newcommand\resumetoc{\let\addcontentsline\origcontentsline}

\title{K3 surfaces with finite automorphism group in positive characteristic}
\author{Alexandros Vlachos}
\begin{document}
\begin{abstract}
    We classify K3 surfaces with finite automorphism group in positive characteristic. The classification is the same as over the complex numbers, except when the base field is $\bFFp$.
\end{abstract}
\maketitle
\tableofcontents

\section{Introduction}
Over the complex numbers, K3 surfaces with finite automorphism group are classified according to their Picard group, which equipped with the intersection pairing is a lattice. The classification has been carried out by Pjatecki\u i-\v Sapiro and \v Safarevi\v c in \cite{pjateckisapiroshafarevic} for Picard ranks $\rho=1,2$, Nikulin in \cite{Nikulin1983, MR752938} for $\rho=3$ and $\rho \geq 5$, and Vinberg in \cite{vinberg} for $\rho=4$. The key observation is that for a complex K3 surface $X$ the automorphism group $\Aut(X)$ is isomorphic to $\mathrm{O}(\Pic(X))/W$ up to a finite group, where $W$ is the Weyl group of $\Pic(X)$ generated by reflections along $-2$ vectors. Thus, the finiteness of $\Aut(X)$ is determined by the finiteness of the factor group $\mathrm{O}(\Pic(X))/W$. One proceeds by classifying the even, hyperbolic lattices $\Lambda$ for which the quotient $\mathrm{O}(\Lambda)/W$ is finite and then determining which ones appear as Picard groups of complex K3 surfaces. For $\rho=1,2$ there are infinitely many isomorphism classes of Picard lattices, while for $\rho \geq 3$ the output of the Nikulin-Vinberg classification is a list of 118 Picard lattices. The full list of lattices can be found in Table \ref{table1} in the \hyperref[table]{Appendix}. Recently, Roulleau in \cite{roul} has given explicit geometric constructions for most of the 118 families of $\rho \geq 3$. 

The aim of this paper is to extend the classification of K3 surfaces with finite automorphism group to positive characteristic. Our main result is the following:

\begin{theorem}\label{mainthm}
    Let $k$ be an algebraically closed field of positive characteristic.
    \begin{enumerate}[left= 0pt]
        \item\label{mainthmitem1} A K3 surface $X$ over $k$ has finite automorphism group if and only if $\Pic(X)$ is in Table \ref{table1}.
        \item\label{mainthmitem2} If $k\neq \bFFp$, then for each lattice $\Lambda$ in Table \ref{table1} there exists a K3 surface $X$ over $k$ with $\Pic(X)\cong\Lambda$.
    \end{enumerate}
\end{theorem}

In particular, the classification in positive characteristic is exactly the same as over the complex numbers, at least when $k \neq \bFFp$. However, the second assertion of Theorem \ref{mainthm} is not true over $\bFFp$, as the Tate conjecture implies that the Picard rank of K3 surfaces over $\bFFp$ is even (see \cite[Corollary 17.2.9]{huyk3}). This is in contrast to the situation in characteristic $0$, where the corresponding assertion is true over $\bar{\mathbb{Q}}$, see \cite[Remark 17.2.16]{huyk3} for a discussion.

Let us give a summary of the arguments and explain the structure of the paper. 

In Section \ref{reduction-section}, we establish the first assertion of Theorem \ref{mainthm} by combining the lifting results from \cite{lieblichmaulik} with the known classification over the complex numbers.

The rest of this paper is devoted to proving the second assertion of Theorem \ref{mainthm}. Note that it is enough to prove it for $k=\overline{\FFp(t)}$, and the general case will follow from base change. Over the complex numbers, the corresponding assertion is proved via the Global Torelli Theorem, which is not available in positive characteristic. We overcome this in roughly two steps:

In Section \ref{keyresult-section} we establish the key technical result needed. Using various deformation arguments, we prove that for a lattice $\Lambda$, the geometric generic fiber of a suitable one-parameter deformation of an ordinary $\Lambda$-quasipolarized K3 surface over $\bFFp$ has Picard group isometric to $\Lambda$. We refer to Proposition \ref{main result} for a precise statement.

In view of the above result, it suffices to construct ordinary $\Lambda$-quasipolarized K3 surfaces over $\bFFp$ for each $\Lambda$ in Table \ref{table1}. This is carried out in Section \ref{construction-section}. The important observation is that every lattice in Table \ref{table1} embeds primitively into the lattice $U\oplus E_8^{\oplus 2}\oplus A_1$, which implies that we only need to construct one ordinary $U\oplus E_8^{\oplus 2}\oplus A_1$-quasipolarized K3 surface for each prime characteristic. For this, we use K3 surfaces of Picard rank 20 defined over number fields and reduction modulo primes.

In Section \ref{automorphism section}, we briefly comment on the general structure of the automorphism groups, and compute the automorphism group of a K3 surface with Picard lattice $U\oplus E_8 ^{\oplus 2} \oplus A_1$.
\begin{acknowledgments}
This paper is an extended version of my master's thesis at the University of Bonn. I would like to express my gratitude to my master's thesis advisor Gebhard Martin for suggesting this problem, for his continued patience and support, and for many helpful discussions, comments and feedback. I would also like to thank Matthias Schütt for helpful discussions and comments on this manuscript. 

ChatGPT was used to assist with parts of the drafting process. The mathematical ideas were developed without AI assistance.

Part of this work was funded by the Deutsche Forschungsgemeinschaft — RTG2965 — Project number 512730679.
\end{acknowledgments}

\newpage

\section{Preliminaries}

\stoptoc

\subsection{Generalities on K3 surfaces}
We begin with some basic definitions and properties of K3 surfaces. Most of the material in this section is adapted from \cite{huyk3}.

\begin{definition}
    A K3 surface $X$ is a smooth, projective, integral surface over an algebraically closed field such that $H^1(X,\mathcal{O}_X)=0$ and $\omega_X\cong\mathcal{O}_X$.
\end{definition}

\begin{proposition}
    Let $X$ be a K3 surface. Then $\Pic(X)=\NS(X)=\mathrm{Num}(X)$.
\end{proposition}

\begin{cor}
    Let $X$ be a K3 surface. Then $\Pic(X)$, equipped with the intersection form, is an even lattice of signature $(1,\rho(X) -1)$, where $\rho(X)$ is the rank of $\Pic(X)$.
\end{cor}

\begin{definition}
    Let $S$ be a scheme. A K3 surface over $S$ is a scheme $\calX$ over $S$ such that the structure morphism $\pi \colon \calX \to S$ is proper and smooth and all geometric fibers are K3 surfaces.
\end{definition}

\begin{definition}
    Let $S$ be a scheme. A K3 space over $S$ is an algebraic space $\calX$ over $S$ such that the structure morphism $\pi \colon \calX \to S$ is proper and smooth and all geometric fibers are K3 surfaces.
\end{definition}

\begin{notation}
    Throughout, $k$ will denote an algebraically closed field of positive characteristic $p$.
\end{notation}

Over fields of positive characteristic, K3 surfaces admit extra properties.
We first define the formal Brauer group of a K3 surface.

Let $X$ be a K3 surface over $k$. Let $\mathrm{Art}_k$ denote the category of local Artinian $k$-algebras with residue field $k$. Consider the functor \[\mathbf{Br}_X \colon \mathrm{Art}_k \to (\mathrm{Ab}), \ A \mapsto \mathrm{Ker}\,(\mathrm{Br}(X_A)\to\mathrm{Br}(X)),\]
where $\mathrm{Br}(X)\coloneqq H^2_{\acute{e}t}(X,\mathbb{G}_m)$ is the cohomological Brauer group of $X$ and similarly for $X_A$. 

\begin{proposition}
    Let $X$ be a K3 surface over $k$. Then $\mathbf{Br}_X$ is prorepresentable by a formally smooth, local, complete, one-dimensional $k$-algebra $R$ with residue field $k$.
\end{proposition}

Since the functor $\mathbf{Br}_X$ takes values in the category of abelian groups $(\mathrm{Ab})$, the formal scheme $\mathrm{Spf}(R)$ has a group structure. 

\begin{definition}
    With the above notation, the formal Brauer group $\widehat{\mathrm{Br}}_X$ of $X$ is $\mathrm{Spf}(R)$.
\end{definition}

\begin{definition}
    The height of a K3 surface $X$ over $k$ is defined as the height $h(X)\coloneqq h(\widehat{\mathrm{Br}}_{X})$ of its formal Brauer group. 
\end{definition}

The height of a K3 surface can alternatively be defined using crystalline cohomology. Even though we will not pursue this approach here, both definitions are equivalent. The crystalline cohomology definition shows that the height of a K3 surface $X$ satisfies $h(X)=1, \dots,10$ or $h(X)=\infty$. A K3 surface $X$ with $h(X)\leq 10$ will be called \emph{of finite height}.

\begin{proposition}\label{heightinequality}
    Let $X$ be a K3 surface of finite height over $k$. Then $\rho(X)\leq 22-2h(X)$.
\end{proposition}

\begin{definition}
    A K3 surface $X$ over $k$ is called ordinary if it satisfies one of the following equivalent conditions:
    \begin{enumerate}
        \item $h(X)=1$.
        \item The Frobenius endomorphism of $H^2(X,\mathcal{O}_X)$ is a bijection.
    \end{enumerate}
\end{definition}

\begin{definition}
    A K3 surface $X$ over $k$ is called supersingular if it satisfies one of the following equivalent conditions:
    \begin{enumerate}
        \item $h(X)=\infty$ (Artin-supersingular).
        \item $\rho(X)=22$ (Shioda-supersingular).
    \end{enumerate}
\end{definition}

\begin{remark}
    By Proposition \ref{heightinequality}, K3 surfaces over $k$ with $\rho=19, 20$ are ordinary.
\end{remark}

\subsection{The specialization map}

We introduce the specialization map (\cite[X, App. 7]{sga6}), an important tool that will be used repeatedly in the following sections. Our exposition closely follows \cite{poonenmaulikns}.

\begin{proposition}\cite[Proposition 3.3]{poonenmaulikns}\label{specializationmap}
    Let $R$ be a DVR with fraction field $K$ and residue field $k'$. Fix an algebraic closure $\overline{K}$ of $K$. Choose a nonzero prime ideal $\mathfrak{p}$ of the integral closure $\overline{R}$ of $R$ in $\overline{K}$, so $\overline{k'}\coloneqq \overline{R}/\mathfrak{p}$ is an algebraic closure of $k'$. Let $\calX$ be a smooth proper $R$-scheme. Denote by $\calX_{\bar{\eta}}$ and $\calX_{\bar{0}}$ the geometric generic and the geometric special fiber, respectively. Then there is a natural homomorphism \[\overline{\mathrm{sp}}\colon \Pic(\calX_{\bar{\eta}})\to \Pic(\calX_{\bar{0}})\] depending only on the choices above. Moreover, if $\overline{\mathrm{sp}}$ maps a class $[\mathcal{L}]$ to an ample class, then $\mathcal{L}$ is ample.
\end{proposition}

\begin{proof}
    If $\mathcal{L}$ is a line bundle on $\calX_{\eta}$ whose image in $\Pic(\calX_0)$ is ample, then the corresponding line bundle on $\calX$ is ample relative to $\Spec(R)$ by \cite[Th\'eor\`eme 4.7.1]{ega3}, so $\mathcal{L}$ is also ample.
    As in the proof of \cite[Theorem 8.4.3]{neronmodels}, we have \[\label{sphom} \tag{$*$}\Pic(\calX_{\eta})\xleftarrow{\sim} \Pic(\calX) \to \Pic(\calX_{0}).\] For each finite extension $L$ of $K$ in $\overline{K}$, the integral closure $R_L$ of $R$ in $L$ is a Dedekind domain by the Krull-Akizuki theorem and localizing at $\mathfrak{p}\cap R_L$ gives a discrete valuation ring $R'_L$. Taking the direct limit of the analogue of (\ref{sphom}) for $R'_L$ yields $\overline{\mathrm{sp}}\colon \Pic(\calX_{\bar{\eta}})\to \Pic(\calX_{\bar{0}})$.
\end{proof}

\begin{remark}
    The previous proposition is also valid when $\calX$ is an algebraic space via the same proof, see \cite[Tag 0DNG]{stacks-project}.
\end{remark}

\begin{remark}
    When $\calX \to \Spec(R)$ is a K3 space, the specialization map $\overline{\mathrm{sp}}$ is also injective and compatible with the intersection forms. This is a consequence of the local constancy of the Euler characteristic for flat morphisms and the fact that the Picard group of a K3 surface is a lattice.
\end{remark}

\resumetoc

\section{Finite automorphisms and Picard lattices}\label{reduction-section}

In this section, we prove the first assertion of Theorem \ref{mainthm}. We begin with the converse implication. As it turns out, the key observation mentioned in the introduction is also valid for most cases in positive characteristic.

\begin{theorem}\cite[Theorem 6.1, Remark 6.3]{lieblichmaulik}\label{observationp}
    Let $X$ be a K3 surface over $k$ if $\mathrm{char}(k)\geq 3$, or a K3 surface of finite height if $\mathrm{char}(k)=2$. Then the natural map $\Aut(X) \to \mathrm{O}(\Pic(X))/W$ has finite kernel and cokernel, where $W$ is the subgroup of $\mathrm{O}(\Pic(X))$ generated by reflections along $-2$-curves.  
\end{theorem}

\begin{proposition}\label{converseimplication}
    Let $X$ be a K3 surface over $k$ such that $\Pic(X)$ is a lattice in Table \ref{table1}. Then $\Aut(X)$ is finite.
\end{proposition}

\begin{proof}
    Every lattice in Table \ref{table1} is of rank $\leq 19$, so $X$ is of finite height. The result then follows from Theorem \ref{observationp} and the fact that $\Pic(X)$ is such that $\mathrm{O}(\Pic(X))/W$ is finite.
\end{proof}

It remains to prove the forward implication of the first assertion of Theorem \ref{mainthm}.

\begin{proposition} \label{forwardimplication}
        Let $X$ be a K3 surface with finite automorphism group over $k$. Then $\Pic(X)$ is one of the lattices in Table \ref{table1}.
\end{proposition}

The proof of this proposition is a formal consequence of the following two results from \cite{lieblichmaulik}.

\begin{theorem}
\cite[Theorem 2.1]{lieblichmaulik} \label{lbthm 2.1}
        Let $\pi \colon \calX \to \Spec(R)$ be a smooth projective relative surface over a Henselian DVR with algebraically closed residue field $k$ and fraction field $K$, whose special fiber is not birationally ruled. Write \[\overline{\mathrm{sp}} \colon \Pic(\calX_{\bar{\eta}}) \to \Pic(\calX_{0})\] for the natural specialization map. If $\overline{\mathrm{sp}}(\mathcal{L})$ is ample for every ample invertible sheaf $\mathcal{L}$ on $\calX_{\Bar{\eta}}$, then there is a group homomorphism $$\bar{\sigma} \colon \Aut(\calX_{\Bar{\eta}}) \to \Aut(\calX_{0})$$ such that $\overline{\mathrm{sp}}$ is $\bar{\sigma}$-equivariant with respect to the natural pullback actions. In addition, if $H^{0}(\calX_0,T_{\calX_0})=0$, then $\bar{\sigma}$ is injective.
\end{theorem}

\begin{proposition}\cite[Corollary 4.2]{lieblichmaulik} \label{lbcor 2.1}
        Any K3 surface $X$ of finite height over $k$ is the closed fiber of a smooth projective relative K3 surface $\mathcal{X} \to \Spec(W(k))$, where $W(k)$ is the ring of Witt vectors with residue field $k$, such that the restriction map $\Pic(\mathcal{X}) \to \Pic(X)$ is an isomorphism.
\end{proposition}

Before proceeding with the proof, let us observe the following:

\begin{lemma}\label{splemma}
    Let $\pi \colon \calX \to \Spec(R)$ be a K3 surface over a Henselian DVR $R$ with residue field $k$. If the restriction map $\Pic(\calX) \to \Pic(\calX_0)$ is an isomorphism, then the specialization map $\overline{\mathrm{sp}} \colon \Pic(\calX_{\bar{\eta}})\to \Pic(\calX_0)$ is an isomorphism.
\end{lemma}

\begin{proof}
By the definition of the specialization maps as in the proof of Proposition \ref{specializationmap} we have the following commutative diagram:
\[\begin{tikzcd}[sep=small]
	{\Pic(\calX_{\eta})} \\
	& {\Pic(\calX_{\bar{\eta}})} \\
	{\Pic(\calX_0)}
	\arrow[from=1-1, to=2-2]
	\arrow["\mathrm{sp}", from=1-1, to=3-1]
	\arrow["\overline{\mathrm{sp}}"', from=2-2, to=3-1]
\end{tikzcd}\]
        where the top right map is given by pullback, which is injective because $\mathbf{Pic}_{\calX \slash R}$ is an fppf sheaf. Since $\mathrm{sp}$ is the composition of $\Pic(\calX_{\eta})\cong \Pic(\calX)$ with the restriction map $\Pic(\calX)\xrightarrow{\sim} \Pic(\calX_0)$, it is an isomorphism and hence $\overline{\mathrm{sp}}$ is too.
\end{proof}

In order to apply Proposition \ref{lbcor 2.1} and Theorem \ref{lbthm 2.1} we need to know that K3 surfaces with finite automorphism group are of finite height. Although this is certainly known to experts, we include a short proof here for lack of a better reference.

\begin{proposition}\label{supersingularauto}
    Let $X$ be a supersingular K3 surface over $k$. Then $\Aut(X)$ is infinite.
\end{proposition}

\begin{proof}
    If $\mathrm{char}(k)$ is odd, then the natural map $\Aut(X) \to \mathrm{O}(\Pic(X))/W$ has finite cokernel by Theorem \ref{observationp}. This proves the claim, since $\mathrm{O}(\Pic(X))/W$ is infinite, as Nikulin's list of even, hyperbolic lattices $\Lambda$ with $\mathrm{O}(\Lambda)/W$ finite does not contain any lattices of rank at least $20$.

    Assume now that $\mathrm{char}(k)=2$. If the Artin invariant of $X$ satisfies $\sigma \leq 9$, then $X$ admits an elliptic fibration with a section and the claim follows from \cite[Corollary 2.5]{hito23}. If $\sigma=10$, then $X$ admits an elliptic fibration with no reducible singular fibers \cite[Theorem 4.8]{kondoshimadaduality}. The associated Jacobian fibration $J(X)$ has isomorphic fibers, hence by the Shioda--Tate formula its Mordell--Weil group is infinite. As $X_{\eta}$ is a $J(X_{\eta})$-torsor, this implies the claim.
\end{proof}

\begin{proof}[Proof of Proposition \ref{forwardimplication}]
        Since $X$ has finite automorphism group, it is of finite height by Proposition \ref{supersingularauto}. Hence, by Proposition \ref{lbcor 2.1} there exists a K3 surface $\calX \to \Spec(W(k))$ such that $\calX_0\cong X$ and $\Pic(\calX) \to \Pic(X)$ is an isomorphism. Lemma \ref{splemma} implies that $\overline{\mathrm{sp}}$ is an isomorphism, and by \cite[Corollary 2.4]{lieblichmaulik} $\overline{\mathrm{sp}}$ preserves the ample cones. Since K3 surfaces have no global vector fields (\cite[Thm. 7]{rudakov}), $H^0(X,T_{X})=0$, so $\calX \to \Spec(W(k))$ satisfies the assumptions of Theorem \ref{lbthm 2.1}. In particular there exists an injective map $\bar{\sigma} \colon \Aut(\calX_{\bar{\eta}}) \to \Aut(X)$, so $\calX_{\bar{\eta}}$ is a K3 surface over an algebraically closed field of characteristic 0 with finite automorphism group. By the classification in characteristic 0, $\Pic(\calX_{\bar{\eta}})$ is then one of the lattices in Table \ref{table1}.
\end{proof}

\section{Deformations of lattice-primitive K3 surfaces}\label{keyresult-section}

In this section, we establish the key tool to prove the existence part of Theorem \ref{mainthm}. In order to state the result, we need to define the notion of a lattice quasipolarization.

\begin{definition}
    A $\Lambda$-primitive K3 surface is a pair $(X,j)$, where $X$ is a K3 surface and $j \colon \Lambda \hookrightarrow \Pic(X)$ is a primitive lattice embedding. If, moreover, $j(\Lambda)$ contains a big and nef line bundle, then we say that $(X,j)$ is $\Lambda$-quasipolarized.
\end{definition}

\begin{definition}
    A $\Lambda$-primitive K3 space is a pair $(\calX \slash S,\mathcal{J})$, where $\calX \to S$ is a K3 space over a scheme $S$, and $\mathcal{J}\colon \Lambda \to \mathbf{Pic}_{\calX \slash S}(S)$ is a group homomorphism such that for all geometric points $\overline{s} \to S$, the induced map $\mathcal{J}_{\overline{s}}\colon \Lambda \to \Pic(\calX_{\overline{s}})$ is a primitive lattice embedding. If, moreover, there exists a $\lambda \in \Lambda$ such that $\mathcal{J}_{\overline{s}}(\lambda)$ is big and nef for every geometric point $\overline{s} \to S$, then we say that $(\calX \slash S, \mathcal{J})$ is $\Lambda$-quasipolarized.
\end{definition}

We can now state the main result of this section.

\begin{proposition}\label{main result}
    Let $(X,j)$ be an ordinary $\Lambda$-quasipolarized K3 surface over $\bFFp$. Then there exists a Henselian DVR $A$ over $\bFFp$, with residue field $\bFFp$, whose fraction field satisfies $\overline{\mathrm{Frac}(A)} \cong \overline{\mathbb{F}_p(t)}$ and a K3 space $\mathcal{X}$ over $\Spec(A)$ such that $\mathcal{X}_0\cong X$ and $\Pic(\mathcal{X}_{\overline{\eta}})\cong\Lambda$.
\end{proposition}

We begin by recalling the deformation theory of lattice-primitive K3 surfaces, which will be an essential ingredient in the proof.

For $(X,j)$ a $\Lambda$-primitive K3 surface, consider the deformation functor \[\mathrm{Def}_{(X,\Lambda,j)} \colon \mathrm{Art}_k  \to (\mathrm{Set})\] parametrizing isomorphism classes of triples $(\calX, J, \phi)$, where $\calX \to \Spec(A)$ is a relative K3 surface, $\phi\colon \calX_0 \cong X$ is an isomorphism, and $J \colon \Lambda \hookrightarrow \Pic(\calX)$ is such that $J_0= \phi^* \circ j$. 

Observe that if one chooses a basis $L_1$, \dots,$L_{\mathrm{rk}\, \Lambda}$ of $\Lambda$, then $\Def_{(X,\Lambda,j)}$ is identified with $\Def_{(X,j(L_1),\dots, j(L_{\rk \Lambda}))}$, the functor which parametrizes deformations of $X$ together with the line bundles $j(L_1), \dots, j(L_{\rk \, \Lambda})$. This can be used to show the following.

\begin{theorem}\cite[Proposition 4.1, Lemma 4.3]{lieblichmaulik} \label{prorep}
    Let $(X,j)$ be a $\Lambda$-primitive K3 surface of finite height over $k$. The deformation functor $\mathrm{Def}_{(X,\Lambda,j)}$ is prorepresented by a formally smooth, local, complete $k$-algebra $D$ of dimension $20-\mathrm{rk}\,\Lambda$ and residue field $k$.
\end{theorem}

\begin{setup}
    Let us explain our setup and fix some notation. For $(X,j)$ a $\Lambda$-primitive K3 surface of finite height over $\bFFp$, $D$ will denote the ring prorepresenting $\mathrm{Def}_{(X,\Lambda,j)}$ and $m$ its maximal ideal. The proof of Proposition \ref{main result} is immediate if $\mathrm{rk}\, \Lambda=\mathrm{rk}\, \Pic(X)$, so for the rest of this section we assume that $\mathrm{rk}\, \Lambda<\mathrm{rk}\, \Pic(X)$. We call a lattice $\Lambda'$ \emph{intermediate} if $j(\Lambda) \subsetneq \Lambda' \subseteq \Pic(X)$ and each embedding is primitive. By definition $X$ is then $\Lambda'$-primitive for any intermediate $\Lambda'$ and we denote by $D'$ the ring prorepresenting $\Def_{(X,\Lambda')}$ and $m'$ its maximal ideal. Since the $\Lambda'$-primitive structure of $X$ is compatible with $j$, there is a natural forgetful transformation $\mathrm{Def}_{(X,\Lambda')}\to \mathrm{Def}_{(X,\Lambda,j)}$, which by \cite[Proposition 1.5]{delignerel} is injective. In particular the associated adic ring map $D\to D'$ is surjective and we can regard the formal deformation space $\Spf(D')$ as a formal closed subscheme of $\Spf(D)$. 
\end{setup}

The strategy to prove Proposition \ref{main result} is to deform $X$ inside $\Spf(D)$ and away from all the $\Spf(D')$ as $\Lambda'$ ranges over all intermediate lattices. For this, we show that there exists a tangent vector pointing away from these loci.

\begin{lemma}\label{tangentunion}
    Let $(X,j)$ be a $\Lambda$-primitive K3 surface of finite height over $\bFFp$. Then \[\bigcup_{\Lambda'} T_{D',m'} \subsetneq T_{D,m},\] where the union is over all intermediate lattices $\Lambda'$.
\end{lemma}

\begin{proof}
    Since $X$ is projective over $\bFFp$ and $\Pic(X)$ is finitely generated we can assume that both $X$ and $\Pic(X)$ are defined over a finite field $\FFpn$. In particular there exists a K3 surface $X_0$ over $\FFpn$ and an isomorphism $X_0 \times_{\FFpn} \bFFp \cong X$ such that pullback induces $\Pic(X_0) \cong \Pic(X)$.
    
    Write $j_0 \coloneqq (\Pic(X_0) \to \Pic(X))^{-1} \circ j$ and note that $(X_0,j_0)$ is $\Lambda$-primitive. Let $\mathrm{Def}_{(X_0,\Lambda,j_0)}$ and $\mathrm{Def}_{(X_0,\Lambda')}$ be the corresponding deformation functors for $\Lambda$ and all intermediate $\Lambda'$, and denote by $(D_0, m_0)$ and $(D'_0,m'_0)$ the local rings prorepresenting them. 
    
    By \cite[Lemma 4.3]{lieblichmaulik} the dual of the tangent space of $\mathrm{Def}_{(X_0,\Lambda,j_0)}$ is naturally identified with the cokernel of $c_{1,X_0}^{\mathrm{Hodge}}\circ (j_0 \otimes_{\mathbb{Z}} \FFpn)$, where $c_{1,X_0}^{\mathrm{Hodge}} \colon \Pic(X_0)\otimes_{\mathbb{Z}}\FFpn \to H^1(X_0, \Omega_{X_0}^1)$ is the first Hodge Chern class map, and accordingly for $\mathrm{Def}_{(X_0,\Lambda')}$, $\mathrm{Def}_{(X,\Lambda,j)}$, $\mathrm{Def}_{(X,\Lambda')}$. Since $X$, $X_0$ are of finite height, the maps $c_1^{\mathrm{Hodge}}$ are injective by \cite[Proposition 10.3]{katsuravandergeerstratifitcation}. As $X$ is the base change of $X_0$ along $\bFFp$ and $j$ is identified with $j_0$ under the pullback $\Pic(X_0) \cong \Pic(X)$, one sees that there exist $\bFFp$-vector space isomorphisms $T_{D_0,m_0}\otimes_{\FFpn} \bFFp \cong T_{D,m}$ and $T_{D'_0,m'_0}\otimes_{\FFpn} \bFFp \cong T_{D',m'}$, which are compatible with all inclusions $T_{D'_0,m'_0} \subseteq T_{D_0,m_0}$ and $T_{D',m'}\subseteq T_{D,m}$. The primitivity of the embeddings of $\Lambda, \Lambda'$, together with $\mathrm{rk} \, \Lambda < \mathrm{rk}\, \Lambda'$ and the injectivity of $c_1 ^{\mathrm{Hodge}}$, implies that the above inclusions are strict.  Since $T_{D_0,m_0}$ is a finite-dimensional vector space over the finite field $\FFpn$, it has finitely many proper linear subspaces. In particular only finitely many $T_{D'_0,m'_0}$ occur. But $\bFFp$ is an infinite field, which implies that $T_{D,m}$ cannot be covered by the finitely many proper linear subspaces $T_{D',m'}$, establishing the lemma.
\end{proof}

Using Lemma \ref{tangentunion} we can now construct the curve inside $\Spf(D)$ we will deform over.

\begin{proposition}\label{compDVRlem}
    Let $X$ be a $\Lambda$-primitive K3 surface of finite height over $\bFFp$. There exists a complete DVR $R$ over $\bFFp$, with residue field $\bFFp$, and an adic surjection $D\to R$, satisfying the following property:
    \begin{description} 
        \item[(DF)] The map $D \to R \to R/m_R ^2$ does not factor through $D \to D'$ for any intermediate $\Lambda'$.
    \end{description}
\end{proposition}

\begin{proof}
    It is enough to construct such a DVR $R$ satisfying $T_{R,m_{R}} \nsubseteq T_{D',m'}$ instead, for all intermediate $\Lambda'$. By Lemma \ref{tangentunion} there exists a vector $v\in T_{D,m}$ with $v \notin T_{D',m'}$ for each $D'$. Consider $v^{\perp} \coloneqq \{ f\in m/m^2 \, | \, v(f)=0 \}$, which is a vector subspace of $m/m^2$ of codimension 1. Choose a basis $\bar{f}_1,\dots,\bar{f}_{d-1}$ of $v^{\perp}$, where $d\coloneqq\mathrm{dim}(D)=20-\rk \, \Lambda$, lift the basis elements to $D$ and let $I=(f_1,\dots,f_{d-1})$ be the ideal they generate. It is clear that the $f_i$ form a regular sequence and the quotient $R\coloneqq D/I$ is regular, with $T_{R,m_R}=\langle v \rangle$. 
\end{proof}

The key part of the proof of Proposition \ref{main result} is showing that $\Pic(\mathcal{X}_{\overline{\eta}})=\Lambda$ and not larger for the candidate $\calX$. The essential ingredient is the following lemma, which is a direct consequence of \cite[Proposition 1.16]{ogusk3cyrstals}. Here, we give an alternative deformation-theoretic argument.

\begin{lemma}\label{tfreelemma}
Let $\calX \to \Spec(R)$ be a K3 space over a complete DVR $R$ of equicharacteristic $p$ with residue field $k$. Assume that $\calX_0$ is ordinary. Then the cokernel of the specialization map $\mathrm{sp}\colon  \Pic(\calX_{\eta}) \to \Pic(\calX_0)$ is torsion-free.
\end{lemma}
    
\begin{proof}
    Equivalently, we show that the restriction map $\Pic(\calX)\to \Pic(\calX_0)$ has torsion-free cokernel. Let $L\in \Pic(\calX_0)$, and suppose that there exists $\mathcal{M}\in \Pic(\calX)$ such that $\mathcal{M}\vert_{\calX_0}=L^{\otimes r} \in \Pic(\calX_0)$ for some $r>0$. We must show that there exists $\mathcal{L}\in \Pic(\calX)$ such that $\mathcal{L}\vert_{\calX_0}=L$. 
    
    We first show that $L$ lifts to a formal line bundle on the formal completion of $\calX$ along $\calX_0$. Let $m$ be the maximal ideal of $R$. Define $\calX_n \coloneqq \calX \times_{R} R/m^{n+1}$ and $\mathcal{L}_0\coloneqq L\in\Pic(\calX_0)$. Assume that we have defined $\mathcal{L}_n \in \Pic(\calX_n)$ lifting $L$. Let $I_{n+1} \coloneqq m^{n+1}\mathcal{O}_{\calX_{n+1}}=\mathrm{ker}(\mathcal{O}_{\calX_{n+1}} \to \mathcal{O}_{\calX_{n}})$ and observe that there exists an isomorphism of $\mathcal{O}_{\calX_0}$-modules $I_{n+1}\cong \mathcal{O}_{\calX_0} \otimes_k \frac{m^{n+1}}{m^{n+2}}$. The obstruction $o_n(\mathcal{L}_n)$ to deforming $\mathcal{L}_n$ over $\calX_{n+1}$ lies in $H^2(\calX_{n+1}, I_{n+1})\cong H^2(\calX_0, \mathcal{O}_{\calX_0}) \otimes_k \frac{m^{n+1}}{m^{n+2}}$, where  $o_n$ is the coboundary map induced by the short exact sequence of sheaves of abelian groups 
    \[0 \longrightarrow I_{n+1} \xrightarrow{s\mapsto 1+s} \mathcal O_{\calX_{n+1}}^\times \longrightarrow \mathcal O_{\calX_n}^\times \longrightarrow 1.\] 
    The vanishing of $H^1(\calX_0,\mathcal{O}_{\calX_0})$ implies that the maps $\Pic(\calX_{n+1}) \to \Pic(\calX_n)$ are injective, hence so is $\Pic(\calX_n) \to \Pic(\calX_0)$ for all $n>0$. As both $\mathcal{L}_n ^{\otimes r}$ and $\mathcal{M}\vert_{\calX_n}$ are lifts of $L^{\otimes r}$ they are therefore isomorphic, hence $r \cdot o_n(\mathcal{L}_n)=o_n(\mathcal{M}\vert_{\calX_n})=0$. We now consider the following cases:
    
    $\bullet \ p\coloneqq\mathrm{char}(k) \nmid r$: Since $H^2(\calX_{n+1},I_{n+1})$ is a $k$-vector space, it has only $p$-torsion. This implies that $o_n(\mathcal{L}_n)=0$ in this case.
    
    $\bullet \ p\mid r$: It suffices to consider only the case $p=r$. Since we are in the equicharacteristic setting, we can define the morphism of schemes $F_n \colon \calX_{p(n+1)-1} \to \calX_n$ as the identity on topological spaces and on the level of sheaves given by $B/m^{n+1}B \to B$, $x \mapsto x^p$ for $\Spec(B)$ an open affine of $\calX_{p(n+1)-1}$. Set $J_{n+1}\coloneqq m^{p(n+1)}\mathcal{O}_{\calX_{p(n+2)-1}}=\mathrm{ker}(\mathcal{O}_{\calX_{p(n+2)-1}} \to \mathcal{O}_{\calX_{p(n+1)-1}})$ and observe that there exists a commutative diagram of sheaves of abelian groups
    \[
        \begin{tikzcd}
            0 \arrow[r] 
              & I_{n+1}
                \arrow[r,"s\mapsto 1+s"]
                \arrow[d,"\varphi_{n+1}"']
              & \mathcal O_{\calX_{n+1}}^\times
                \arrow[r]
                \arrow[d,"F_{n+1}^{\#}"]
              & \mathcal O_{\calX_n}^\times
                \arrow[r]
                \arrow[d,"F_n^{\#}"]
              & 1 \\
            0 \arrow[r]
              & J_{n+1}
                \arrow[r,"s\mapsto 1+s"]
              & \mathcal O_{\calX_{p(n+2)-1}}^\times
                \arrow[r]
              & \mathcal O_{\calX_{p(n+1)-1}}^\times 
              \arrow[r]
              & 1 ,
        \end{tikzcd}
    \]
    where $\varphi_{n+1}(s)=s^p$. The long exact sequence in cohomology induces a commutative square 
    \begin{equation}\label{obstruction-square}
        \begin{tikzcd}
        \Pic(\calX_n)
            \arrow[r,"o_n"]
            \arrow[d,"F_n ^*"']
          & H^2(\calX_0,I_{n+1})
            \arrow[d,"H^2(\varphi_{n+1})"] \\
        \Pic(\calX_{p(n+1)-1})
            \arrow[r,"O_n"']
          & H^2(\calX_{p(n+2)-1},J_{n+1}),
        \end{tikzcd}
    \end{equation}
    and we claim that $H^2(\varphi_{n+1})$ is injective. Indeed, set $\overline{J_{n+1}}\coloneqq J_{n+1}/m^{p(n+1)+1}\mathcal{O}_{\calX_{p(n+2)-1}}$ and note that there is a natural isomorphism of $\mathcal{O}_{\calX_0}$-modules $\overline{J_{n+1}}\cong \mathcal{O}_{\calX_0} \otimes_k \frac{m^{p(n+1)}}{m^{p(n+1)+1}}$. Under this identification and the analogous one for $I_{n+1}$, the composite $I_{n+1} \xrightarrow{\varphi_{n+1}} J_{n+1} \to \overline{J_{n+1}}$ is given by $a\otimes v \mapsto F^{\#}(a) \otimes v^p$, where $F$ is the absolute Frobenius of $\calX_0$. Hence, the induced map on second cohomology $H^2(\calX_0, \mathcal{O}_{\calX_0}) \otimes_k \frac{m^{n+1}}{m^{n+2}} \to H^2(\calX_0, \mathcal{O}_{\calX_0}) \otimes_k \frac{m^{p(n+1)}}{m^{p(n+1)+1}}$ is given by $a\otimes v \mapsto H^2(F^{\#})(a) \otimes v^p$. But $\calX_0$ is ordinary, so $H^2(F^{\#})$ is bijective. Also $k$ is algebraically closed, hence perfect, which means that $\frac{m^{n+1}}{m^{n+2}} \to \frac{m^{p(n+1)}}{m^{p(n+1)+1}}$, $v \mapsto v^p$ is also bijective. Thus the composite of $H^2(\varphi_{n+1})$ with the map induced by $J_{n+1} \to \overline{J_{n+1}}$ is injective, hence so is $H^2(\varphi_{n+1})$. Now, observe that $F$ factors through $F_0$, from which we can deduce that $(F_n^*\mathcal{L}_n)|_{\calX_0}\cong F^*\mathcal{L}_0\cong \mathcal{L}_0^{\otimes p}\cong \mathcal{M}|_{\calX_0}$. Injectivity of the restriction $\Pic(\calX_{p(n+1)-1})\to \Pic(\calX_0)$ hence implies that $F_n ^* \mathcal{L}_n \cong \mathcal{M}\vert_{\calX_{p(n+1)-1}}$. This means that $O_n(F_n ^* \mathcal{L}_n)=0$, which together with the commutativity of (\ref{obstruction-square}) and injectivity of $H^2(\varphi_{n+1})$ implies that $o_n(\mathcal{L}_n)=0$.

    The vanishing of the obstructions $o_n(\mathcal{L}_n)$ for all $n\geq0$ implies that we can lift $\mathcal{L}_0$ to a formal line bundle on the formal completion of $\calX$ along $\calX_0$. Grothendieck's existence theorem for algebraic spaces \cite[Tag 08BE]{stacks-project} implies that this formal line bundle is induced by a line bundle $\mathcal{L}\in \Pic(\calX)$ on $\calX$. In particular, $\mathcal{L}\vert_{\calX_0}=L$.
\end{proof}

\begin{proof}[Proof of Proposition \ref{main result}]
    Let $\varphi \colon D \to R$ be the map that satisfies the property (DF) from Proposition \ref{compDVRlem}. Let $A$ denote the Henselization of $\bFFp[t]_{(t)}$, which is a DVR by \cite[\href{https://stacks.math.columbia.edu/tag/0AP3}{Tag 0AP3}]{stacks-project}. It then follows that $\mathrm{Frac}(A)$ is a separable algebraic extension of $\overline{\mathbb{F}_p}(t)$, hence $\overline{\mathrm{Frac}(A)} \cong \overline{\mathbb{F}_p(t)}$. Since both $R$ and $\hat{A}$ are complete equicharacteristic DVRs with residue field $\bFFp$, there exist isomorphisms $R \cong \hat{A}\cong \bFFp[[t]]$ as $\bFFp$-algebras. We fix one and for the remainder of this proof replace $R$ by $\hat{A}$.

    Via this identification $\varphi$ induces a formal $\Lambda$-primitive deformation $(\mathfrak{Y}, \mathfrak{J}) \to \Spf(\hat{A})$ of $(X,j)$. As $j$ is a quasipolarization there exists a primitive $\lambda \in \Lambda$ such that $j(\lambda)$ is big and nef, and write $\lambda^2=2d$. Then $(\mathfrak{Y}, \mathfrak{J}(\lambda))$ is a formal object of the stack $M_{2d}$ of primitively quasipolarized K3 spaces, which is algebraic by \cite[Proposition 2.1]{mauliksupersingulark3forlargeprimes}. Hence $(\mathfrak{Y}, \mathfrak{J}(\lambda))$ is effective, given by the formal completion of a K3 space $(\calY,\mathcal{L}) \to \Spec(\hat{A})$, with $\mathcal{L}\in \Pic(\calY)$ inducing $\mathfrak{J}(\lambda)$. By Grothendieck's existence theorem for algebraic spaces \cite[Tag 08BE]{stacks-project} we can algebraize the rest of the formal line bundles induced by $\mathfrak{J}$ in order to obtain the $\Lambda$-quasipolarized K3 space $(\calY \slash \hat{A}, \mathcal{J})$, with $\mathcal{J}_0=j$.

    The moduli functor of $\Lambda$-quasipolarized K3 spaces is locally of finite presentation. Hence Artin approximation \cite[Theorem I.12]{Artin-Algebraic-Approximation} yields a $\Lambda$-quasipolarized K3 space $(\calX, J) \to \Spec(A)$ and an isomorphism of $\Lambda$-primitive deformations $(\mathcal X,J)\big\vert_{A/m_A^2}\cong(\mathcal Y,\mathcal J)\big\vert_{\widehat A/m_{\widehat A}^2}$.

    We now show that $\Pic(\calX_{\overline{\eta}}) \cong \Lambda$. By definition we have that $J_{\overline{\eta}} \colon \Lambda \hookrightarrow\Pic(\calX_{\overline{\eta}})$ is injective. Assume for contradiction that $J_{\overline{\eta}}(\Lambda) \subsetneq \Pic(\calX_{\overline{\eta}})$.  The Picard scheme $\Pic_{\calX_{\eta}\slash \kappa(\eta)}$ is \'etale over $\kappa(\eta)$, which implies that every line bundle of $\calX_{\overline{\eta}}$ is defined over a finite separable extension of $\kappa(\eta)$. By standard limit arguments we then have $\Pic(\calX_{\overline{\eta}})\cong\Pic(\calX_{K'})$, where $K'$ is a finite separable extension of $\kappa(\eta)$. The integral closure of $A$ in $K'$ is a DVR $A'$, with ramification index $s$ and $\mathrm{Frac}(A')=K'$.
    
    The formal completions of the base changes $\calX_{\hat{A}}$, $\calX_{\hat{A'}}$ of $\calX$ along $A \to \hat{A}\to\hat{A'}$ induce formal $\Lambda$-primitive deformations of $(X,j)$, hence maps $D \xrightarrow{\psi}\hat{A}\to\hat{A'}$. By Lemma \ref{tfreelemma} the specialization map $\mathrm{sp}_{\hat{A'}}\colon \Pic(\calX_{\hat{K'}}) \to \Pic(X)$ is primitive. If $\Lambda' \coloneqq \mathrm{sp}_{\hat{A'}}(\Pic(\calX_{\hat{K'}}))$, then our assumption implies that $j(\Lambda) \subsetneq \Lambda'$, hence $\Lambda'$ is intermediate. Since every line bundle on the generic fiber of $\calX_{\hat{A'}}$ extends over $\hat{A'}$ and its restriction to the closed fiber is its specialization, it follows that $\calX_{\hat{A'}}$ and the formal deformation it induces are $\Lambda'$-primitive. In particular, this means that the ring map $D \xrightarrow{\psi} \hat{A} \to \hat{A'}$ factors through $D \to D'$, which places us in the situation of the following diagram:\[\begin{tikzcd}[cramped]
	& {D'} \\
	D & {\hat{A}} & {\hat{A'}} \\
	{\hat{A}} & {\hat{A}/m_{\hat{A}}^2} & {\hat{A'}/m_{\hat{A'}}^{2s}}
	\arrow["\exists "',from=1-2, to=3-3, in=60, out=0, looseness=1.3, swap, dashed]
	\arrow[from=2-1, to=1-2]
	\arrow["{\psi}", from=2-1, to=2-2]
	\arrow["\varphi"', from=2-1, to=3-1]
	\arrow[from=2-2, to=2-3]
	\arrow[from=2-2, to=3-2]
	\arrow[from=2-3, to=3-3]
	\arrow[from=3-1, to=3-2]
	\arrow[from=3-2, to=3-3]
	\arrow[from=2-1, to=3-2, draw=none, "\circlearrowright" description]
\end{tikzcd}\]
where the left square is commutative due to Artin approximation and the dashed arrow exists. By the injectivity of the bottom-right arrow, the map $D \xrightarrow{\varphi} \hat{A} \to \hat{A}/m_{\hat{A}}^{2}$ then factors through $D \to D'$, which contradicts the property (DF) that $\varphi$ satisfies.
\end{proof}

\section{Constructing \texorpdfstring{$\Lambda$}{L}-quasipolarized K3 surfaces}\label{construction-section}

In this section, we construct $\Lambda$-quasipolarized ordinary K3 surfaces over $\bFFp$ for every $\Lambda$ in Table \ref{table1}. This is enough to prove the existence part of Theorem \ref{mainthm} by applying Proposition \ref{main result} to the constructed surfaces. The proof is given at the end of this section.

\stoptoc
\subsection{Primitive lattice embeddings}   

In order to simplify the construction we first examine the possible primitive embeddings between the lattices of Table \ref{table1}. 

\begin{proposition}\label{simplification}
    Every lattice in Table \ref{table1} embeds primitively into $U\oplus E_8^{\oplus2}\oplus A_1$.
\end{proposition}

The proof of Proposition \ref{simplification} relies on the following two useful results on primitive lattice embeddings. For an even lattice $S$, we denote by $A_S\coloneqq S^*/S$ the discriminant group of $S$ and $l(A_S)$ the minimal number of generators.

\begin{theorem}\cite[Corollary 1.12.3]{Nikulin1980-nm}\label{laticcethm}
Let $S$ be an even lattice of signature $(t_-, t_+)$. Then there exists a primitive embedding of $S$ into an even unimodular lattice $M$ of signature $(l_-,l_+)$ if the following conditions are simultaneously satisfied
\begin{enumerate}
    \item $l_+ -l_- \equiv 0 \pmod 8$.
    \item $t_{\pm} \leq l_{\pm}$.
    \item $l(A_S)<\mathrm{rk}\,M-\mathrm{rk}\,S$.
\end{enumerate}
    \end{theorem}

\begin{theorem}\cite[Theorem 1.12.4]{Nikulin1980-nm}\label{rk9pol}
    Every even lattice of signature $(t_+,t_-)$ admits a primitive embedding into an even unimodular lattice of signature $(l_+,l_-)$ if and only if:
    \begin{enumerate}
        \item $l_+ -l_- \equiv 0 \pmod 8$.
        \item $t_{\pm} \leq l_{\pm}$.
        \item $t_+ + t_- \leq \frac{1}{2}(l_+ + l_-)$.
    \end{enumerate}
\end{theorem}

\begin{proof}[Proof of Proposition \ref{simplification}]
    The claim for the lattices of rank $\leq 9$ follows from Theorem \ref{rk9pol} by embedding into the unimodular lattice $U\oplus E_8^{\oplus2}$. By Theorem \ref{laticcethm} and $U\oplus E_8^{\oplus2}$ again, we can verify by inspecting Table \ref{table2} that the claim holds for every lattice except possibly the following: $U\oplus E_8 \oplus A_1^{\oplus 4}$, $U\oplus E_8 \oplus D_4 \oplus A_1$, $U\oplus E_8\oplus D_6$, $U\oplus E_8 \oplus E_7$, $U\oplus D_8 \oplus A_1 ^{\oplus3}$, $U\oplus D_8 \oplus D_4$, $U\oplus D_4 \oplus A_1 ^{\oplus 5}$, $U\oplus D_4^{\oplus 2} \oplus A_1^{\oplus2}$, $U\oplus A_1^{\oplus8}$.

    For the remaining lattices we argue on a case by case basis. By \cite{oguisoshioda} the lattices $A_1^{\oplus 4}$, $D_4 \oplus A_1$, $D_6$ and $E_7$ can be embedded primitively into $E_8$. Therefore, from the remaining lattices all except $U\oplus D_8 \oplus A_1 ^{\oplus3}$ and $U\oplus D_8 \oplus D_4$ embed primitively into $U\oplus E_8^{\oplus2}$. The first lattice embeds primitively into the second since there exists a primitive embedding of $A_1 ^{\oplus3}$ into $D_4$ by sending the three basis elements of $A_1^{\oplus3}$ into the three orthogonal basis elements of $D_4$. We claim that $U\oplus D_8 \oplus D_4$ embeds primitively into $U\oplus E_8^{\oplus2}\oplus A_1$. Indeed, by inspecting the quadratic forms and using \cite[Corollary 1.6.1]{Nikulin1980-nm}, $D_4$ embeds primitively into $E_8$ with orthogonal complement $D_4$ and $D_8$ embeds primitively into $U\oplus E_8$ with orthogonal complement $U(2)$. Thus, $D_4\oplus D_8$ embeds primitively into $U\oplus E_8^{\oplus 2} \oplus A_1$ with orthogonal complement $U(2)\oplus D_4\oplus A_1$. It suffices to show that the latter lattice contains $U$. By \cite[Theorem 3.6.2]{Nikulin1980-nm} the genus of $2$-elementary lattices is determined by the invariants $(t_+,t_-,\alpha,\delta)$, where $(t_+,t_-)$ is the signature, $\alpha$ is the natural number such that $\vert\mathrm{disc}(\Lambda)\vert=2^\alpha$ and $\delta \in \{0,1\}$ is such that it is equal to $0$ if $x^2 \in 2\mathbb{Z}$ for every $x\in \Lambda^*$, and $1$ otherwise. One sees that the lattices $U(2)\oplus D_4\oplus A_1$ and $U\oplus A_1^{\oplus5}$ have the same invariants $(1,6,5,1)$ and thus are in the same genus. But $U\oplus A_1^{\oplus5}$ is unique in its genus since it contains $U$ (\cite[Corollary 1.13.4]{Nikulin1980-nm}), hence the two lattices are isomorphic and we conclude.
\end{proof}

We end this subsection with the following important observation.

\begin{lemma}\label{primitivepolarization}
    Let $\Lambda$ be an even hyperbolic lattice and $(X,j)$ a $\Lambda$-primitive K3 surface. Then there exists a quasipolarization $\tilde{j}\colon \Lambda \to \Pic(X)$.
\end{lemma}

\begin{proof}
    As $\Lambda$ is hyperbolic, there exists an $h\in \Lambda$ with $h^2>0$, which after possibly multiplying by $-1$ we can assume that $j(h)$ lies in the positive cone of $X$. Since the nef cone is a fundamental domain for the action of $W$ on the positive cone of $X$, there exist rational curves $C_1, \dots, C_n$ such that $s_{C_1}\circ \dots \circ s_{C_n}(j(h))$ is nef, where $s_{C_i}$ denotes the reflection along $[C_i]$. Then the primitive embedding $\tilde{j} \coloneqq s_{C_1}\circ \dots \circ s_{C_n} \circ j $ is a quasipolarization.
\end{proof} 

In particular, the following is immediate.

\begin{cor}\label{latticequasipol}
    Let $X$ be a $U\oplus E_8^{\oplus 2}\oplus A_1$-primitive K3 surface. Then $X$ is $\Lambda$-quasipolarized for every $\Lambda$ in Table \ref{table1}.
\end{cor}

\subsection{Construction}

In view of Corollary \ref{latticequasipol}, we need only exhibit an ordinary $U\oplus E_8^{\oplus 2}\oplus A_1$-primitive K3 surface over $\bFFp$ for every prime $p$. 
\begin{proposition}\label{singular-existence}
    There exists an ordinary $U\oplus E_8^{\oplus2} \oplus A_1$-primitive K3 surface over $\overline{\mathbb{F}}_p$.
\end{proposition}

To prove Proposition \ref{singular-existence}, we will construct for each prime $p$ a suitable $U\oplus E_8^{\oplus2} \oplus A_1$-primitive K3 surface over a number field and then reduce it modulo $p$.

\subsubsection{Singular complex K3 surfaces}
We first recall some special properties of singular complex K3 surfaces, i.e. complex K3 surfaces of Picard rank $20$.

\begin{theorem}\cite[Theorem 4]{Shioda_Inose_1977}\label{si bijection}
    The map $X \mapsto T(X)$ defines a bijection from the set of isomorphism classes of singular complex K3 surfaces to the set of isomorphism classes of positive definite, even, oriented lattices of rank 2.
\end{theorem}

Here, an orientation of a rank $2$ lattice is the choice of an equivalence class of bases, where two bases are equivalent if and only if their change of basis matrix is in $\mathrm{SL}(2,\mathbb{Z})$.

\begin{theorem}\cite[Theorem 6]{Shioda_Inose_1977}\label{si fieldofdefinition}
    Every singular complex K3 surface has a model defined over a number field.
\end{theorem}

In view of Theorem \ref{si fieldofdefinition}, the strategy for proving Proposition \ref{singular-existence} is to first construct a complex $U\oplus E_8^{\oplus 2} \oplus A_1$-primitive K3 surface using Theorem \ref{si bijection} and then reduce it modulo $p$.

\begin{lemma}\label{k3charnot2}
    Let $\alpha$ be a positive integer. There exists a complex K3 surface $X$ with $\Pic(X) \cong U\oplus E_8 ^{\oplus 2} \oplus A_1 \oplus \mathbb{Z}(-2\alpha)$.
\end{lemma}

\begin{proof}
By Theorem \ref{si bijection}, there exists a complex K3 surface $X$ with $T(X)=\mathbb{Z}(2)\oplus\mathbb{Z}(2\alpha)$. The discriminant form of $\Pic(X)$ and $\rho(X)$ uniquely determine its isomorphism type (\cite[Corollary 14.3.6]{huyk3}). Let $\Lambda\coloneqq U\oplus E_8 ^{\oplus 2} \oplus A_1 \oplus \mathbb{Z}(-2
\alpha)$. Since $H^2(X,\mathbb{Z})$ is unimodular, \cite[Corollary 1.6.1]{Nikulin1980-nm} implies that $(A_{\Pic(X)}, q_{\Pic(X)})=(A_{T(X)},-q_{T(X)})=(A_{\Lambda},q_{\Lambda})$ and the lemma follows by the above observation.
\end{proof}

\begin{lemma}\label{k3char2}
    Let $T$ be the even, positive definite lattice given by the Gram matrix 
$\begin{psmallmatrix}
    2 & 1\\
    1 & 4
\end{psmallmatrix}$. There exists a complex K3 surface $X$ with $\Pic(X)\cong U\oplus E_8^{\oplus 2}\oplus T(-1)$.
\end{lemma}

\begin{proof}
    Follows from the same reasoning as Lemma \ref{k3charnot2}.
\end{proof}

\begin{remark}
    The K3 surfaces in Lemma \ref{k3charnot2} and Lemma \ref{k3char2} are $U\oplus E_8^{\oplus 2} \oplus A_1$-primitive. For the ones from Lemma \ref{k3charnot2} the assertion is obvious and for the other one note that $A_1$ embeds primitively into $T(-1)$ by sending the basis element of $A_1$ to the basis element of $T(-1)$ with square $-2$.
\end{remark}

\subsubsection{Reduction of K3 surfaces}
We begin by recalling the definition of (potential) good reduction.

\begin{definition}
    Let $X$ be a K3 surface over a number field $K$ and $v$ be a finite place of $K$ of residue characteristic $p$. Denote by $K_v$ the completion of $K$ at $v$. We say that $X$ has 
    \begin{enumerate}
        \item good reduction at $v$ if there exists a proper smooth algebraic space $\calX$ over $\Spec(\mathcal{O}_{K_v})$ with generic fiber isomorphic to $X\times_K K_v$. 
        \item potential good reduction at $v$ if there exists a finite extension $L$ of $K$ such that $X\times_K L$ has good reduction at an extension of $v$ on $L$.
    \end{enumerate}
\end{definition}

We now recall the notion of K3 surfaces with complex multiplication.

Let $X$ be a complex K3 surface and consider the rational Hodge structure on $T(X)_{\mathbb{Q}}$ inherited from $H^2(X,\mathbb{Q})$. Let $\mathrm{End}_{\mathrm{Hdg}}(T(X)_{\mathbb{Q}})$ be the $\mathbb{Q}$-algebra of $\mathbb{Q}$-linear endomorphisms of $T(X)_{\mathbb{Q}}$ preserving the Hodge structure. One can prove that $\mathrm{End}_{\mathrm{Hdg}}(T(X)_{\mathbb{Q}})$ is a number field (see \cite[Corollary 3.3.6]{huyk3}).

A number field $F$ is called totally real if every embedding $F \hookrightarrow \mathbb{C}$ has image in $\mathbb{R}\subseteq \mathbb{C}$. A purely imaginary quadratic field extension $E$ of a totally real number field $F$ is called a complex multiplication (CM) field.

\begin{definition}
    Let $E$ be a CM field. A complex K3 surface $X$ has complex multiplication given by $E$ if $\mathrm{End}_{\mathrm{Hdg}}(T(X)_{\mathbb{Q}}) \cong E$ and $\mathrm{dim}_E T(X)_{\mathbb{Q}}=1$.
\end{definition}

\begin{proposition}
    A singular complex K3 surface $X$ has complex multiplication given by $\mathbb{Q}(\sqrt{-d})$, where $d \coloneqq \mathrm{disc}(T(X))$.
\end{proposition}
\begin{proof}
See \cite[Proposition 3.1]{valonicomplexmultp}.
\end{proof}

Now, let $X$ be a singular $U\oplus E_8^{\oplus 2} \oplus A_1$-primitive complex K3 surface. By Theorem \ref{si fieldofdefinition} we can assume that $X$ is defined over a number field $K$. Let $v$ be a finite place of $K$ of residue characteristic $p$. In order to prove Proposition \ref{singular-existence}, we require the following:
\begin{itemize}
    \item $X$ has (potential) good reduction at $v$.
    \item The reduction $\calX_{\bar{v}}$ of $X$ at $v$ is ordinary.
    \item The reduction $\calX_{\bar{v}}$ of $X$ at $v$ is $U\oplus E_8^{\oplus 2} \oplus A_1$-primitive.
\end{itemize}

Matsumoto has proven that the first requirement holds in greater generality:
\begin{theorem}\cite[Corollary 0.5]{matsumoto2015}\label{matsumotoreduction}
    Any singular K3 surface has potential good reduction at all finite places of residue characteristic $p \geq 5$.
\end{theorem}

For the second requirement, the CM field of $X$ controls the (non)supersingular reduction of $X$. More precisely, Ito has shown the following:
\begin{theorem}\cite[Theorem 8.3]{itoreduction}\label{supersingularreduction}
    Let $X$ be a singular complex K3 surface defined over a number field $K$ with complex multiplication given by $E$. Let $p$ be a prime. Let $v$ be a finite place of $K$ whose residue characteristic is $p$. Assume that $p$ does not divide $\mathrm{disc}(\Pic(X))$ and $X$ has good reduction at $v$. Let $\calX$ be a smooth model over $\mathcal{O}_{K_v}$. Then the geometric special fiber $\calX_{\bar{v}}$ is supersingular if and only if $p$ does not split in $E$.
\end{theorem}

For the third requirement, observe that it suffices to show that the cokernel of the specialization map $\overline{\mathrm{sp}} \colon \Pic(\calX_{\bar{\eta}}) \to \Pic(\calX_{\bar{v}})$ is torsion-free, i.e. $\overline{\mathrm{sp}}$ is a primitive embedding, since then the reduction inherits the primitive lattice embedding of the geometric generic fiber.

\begin{lemma}[cf. Lemma \ref{tfreelemma}]\cite[Lemma 3.2]{Jang}\label{jangcokernel}
    Let $X$ be a singular K3 surface defined over a number field $K$. Let $p$ be a prime. Let $v$ be a finite place of $K$ of residue characteristic $p$. Assume that $X$ has good reduction at $v$. Let $\calX$ be a smooth model of $X$ over $\mathcal{O}_{K_v}$. Then the cokernel of the specialization map \[\overline{\mathrm{sp}} \colon \Pic(\calX_{\bar{\eta}}) \to \Pic(\calX_{\bar{v}})\] has no prime-to-$p$ torsion.
\end{lemma}

We can now prove Proposition \ref{singular-existence} by combining the previous three results.

\begin{proof}[Proof of Proposition \ref{singular-existence}]
    We first choose, for every prime $p$, a singular $U \oplus E_8^{\oplus 2} \oplus A_1$-primitive K3 surface $X$ over a number field $K$ with good reduction at a finite place $v$ of residue characteristic $p$. For $p>3$, let $\alpha= p-1$ and choose $X$ a K3 surface of Lemma \ref{k3charnot2} for this $\alpha$. By Theorems \ref{si fieldofdefinition} and \ref{matsumotoreduction}, we can assume after base change that $X$ is defined over a number field and has good reduction at a finite place of residue characteristic $p$. As Theorem \ref{matsumotoreduction} does not treat residue characteristics $2$ and $3$, we give explicit smooth models. For $p=2$, the modular surface $S(\Gamma_1(7))$ associated to $\Gamma_1(7)$ is a singular K3 surface over $\mathbb{Q}$ with Picard lattice $U\oplus E_8^{\oplus2} \oplus T(-1)$, where $T$ is the lattice of Lemma \ref{k3char2} (\cite[Proposition 4.5]{modularseven}), and admits a smooth model over $\mathbb{Z}[1/7]$. For $p=3$, the modular surface $S(\Gamma_1(8))$ associated to $\Gamma_1(8)$ is a singular K3 surface over $\mathbb{Q}$ with Picard lattice $U \oplus E_8^{\oplus2} \oplus \mathbb{Z}(-2) \oplus \mathbb{Z}(-4)$ (\cite[Introduction]{modulareight}) and admits a smooth model over $\mathbb{Z}[1/8]$.

    Let $\calX$ be a smooth model of $X$ over $\mathcal{O}_{K_v}$. By our choice of $X$, the prime $p$ splits in the CM field $\mathbb{Q}(\sqrt{-(p-1)})$ (respectively $\mathbb{Q}(\sqrt{-7}) $ for $p=2$). Hence by Theorem \ref{supersingularreduction} $\calX_{\bar{v}}$ is of finite height. The injectivity of the specialization map and Proposition \ref{heightinequality} yield $\rho(\calX_{\bar{v}})=20$. By Lemma \ref{jangcokernel} the cokernel of $\overline{\mathrm{sp}}$ has no prime-to-$p$ torsion. We now show that it has no $p$-torsion. The specialization map yields a lattice embedding $\Pic(\calX_{\bar{\eta}})\hookrightarrow \Pic(\calX_{\bar{v}})$, which is of finite index since they have equal rank. An easy computation shows that \[\vert\mathrm{disc}(\Pic(\calX_{\bar{\eta}}))\vert=\vert\mathrm{disc}(\Pic(\calX_{\bar{v}}))\vert \cdot [\Pic(\calX_{\bar{v}}) \colon \Pic(\calX_{\bar{\eta}})]^2.\] But $\vert\mathrm{disc}(\Pic(\calX_{\bar{\eta}}))\vert=4(p-1)$ ($7$ for $p=2$), which is not divisible by $p$, hence $p$ does not divide the index. In particular $\Pic(\calX_{\bar{\eta}})\cong \Pic(\calX_{\bar{v}})$ via $\overline{\mathrm{sp}}$.
    
\end{proof}

As a corollary, we obtain the second assertion of Theorem \ref{mainthm}.
\begin{cor}
    Assume $k\neq \bFFp$. For each lattice $\Lambda$ in Table \ref{table1}, there exists a K3 surface $X$ over $k$ with $\Pic(X)\cong\Lambda$.
\end{cor}

\begin{proof}
     By Proposition \ref{singular-existence} there exists an ordinary $U\oplus E_8^{\oplus 2}\oplus A_1$-primitive K3 surface $Y$ over $\bFFp$. By Proposition \ref{simplification} and Lemma \ref{primitivepolarization} $Y$ is $\Lambda$-quasipolarized for any $\Lambda$ in Table \ref{table1}. The corollary follows from applying Proposition \ref{main result} to $Y$, which yields $X$ over $\overline{\FFp(t)}$ with $\Pic(X)\cong\Lambda$, and base changing $X$ to $k$. 
\end{proof}

\resumetoc 

\section{A note on the automorphism groups}\label{automorphism section}

With the classification complete, a natural next step is to determine the corresponding automorphism groups. Over the complex numbers, this was addressed by Kond\=o in \cite{kondofiniteauto}, where he computed the automorphism group of a general K3 in a given family. On a K3 surface $X$ the automorphism group $\Aut(X)$ acts naturally by pullback on the Picard group and the space of global $2$-forms. Understanding these two actions is in principle enough to determine $\Aut(X)$, at least over $\mathbb{C}$ via the Global Torelli theorem. Let $H_X\coloneqq\mathrm{ker}\left( \Aut(X) \to \mathrm{O}(\Pic(X)) \right)$ and $\Aut_s(X)\coloneqq\mathrm{ker}\left(\Aut(X) \to \mathrm{GL}(H^0(X,\omega_X))\right)$. It is known that $\Aut(X)$ acts on $H^0(X,\omega_X)$ via a finite cyclic group (\cite[Lemma 2.1]{matsumotoresultonauto}), hence following Kond\=o one could examine the short exact sequence \[1 \to \Aut_s(X) \to \Aut(X) \to \mathbb{Z}/m\mathbb{Z} \to 1.\]
Over $\mathbb{C}$, Hodge theory shows that $H_X \cap \Aut_s(X)=1$. This observation enables Kond\=o to compute $\Aut_s(X)$ in many cases by studying its faithful action on $\Pic(X)$. This approach however fails in positive characteristic, due to the presence of wild automorphisms in $H_X$ violating $H_X \cap \Aut_s(X)=1$. Our first proposition clarifies the structure of $H_X$ in this case.

\begin{proposition}\label{trivialautoprop}
    Let $X$ be a K3 surface with finite automorphism group over $k$. Then $H_X$ is an extension of a cyclic group of order prime to $p$ by a $p$-group.
\end{proposition}

\begin{proof}
    First, define $Q \coloneqq \mathrm{im}\left( H_X \to \mathrm{GL}(H^0(X,\omega_X))\right)$. As noted above, $Q$ is finite and cyclic. The fact that $H^0(X,\omega_X)$ is one-dimensional implies that $Q \subseteq k^{\times}$, hence $p \nmid \vert Q \vert$ since $k^{\times}$ contains no non-trivial $p$-torsion. Let $K \coloneqq \mathrm{ker}\left(H_X \to Q\right)$, the group of symplectic automorphisms acting trivially on $\Pic(X)$. We claim that $K$ is a $p$-group. Observe that if $\alpha\in K$ and $\mathrm{ord}(\alpha)=p^m n$ with $(p,n)=1$, then $\alpha^{p^m}$ is tame, so it suffices to show that $K$ contains no non-trivial tame automorphisms. Suppose for contradiction otherwise, and let $g$ be such. It follows from \cite[Remark 15.1.9]{huyk3} that necessarily $\rho(X) \geq9$. By inspecting the graphs of rational curves in \cite{roul}, we see that $X$ in this case contains the following configuration of smooth rational curves:
\begin{center}
\begin{tikzpicture}[scale=1]
    \draw (-1,0)--(1,0);
    \fill (-1,0) circle (2pt);
    \fill (0,0) circle (2pt);
    \fill (1,0) circle (2pt);
    \node[above] at (-1,0) {$A_2$};
    \node[above] at (0,0) {$A_1$};
    \node[above] at (1,0) {$A_3$};
    \draw (0,0) -- (0,-1);
    \fill (0,-1) circle (2pt);
    \node[below] at (0,-1) {$A_4$};
\end{tikzpicture}.
\end{center}
Since $g$ acts trivially on $\Pic(X)$ and $A_1$ intersects $A_j$, $j\in \{2,3,4\}$ transversely at a point, it fixes these three intersection points on $A_1$. As $A_1$ is isomorphic to $\mathbb{P}^1$, $g$ fixes $A_1$ pointwise. However, finite tame symplectic automorphisms have a finite fixed locus by \cite[Theorem 3.3]{dolgkeumsymplectic}, hence this is a contradiction.
\end{proof}

As an application, we compute $\Aut(X)$ for a K3 surface $X$ with $\Pic(X)\cong U\oplus E_8 ^{\oplus 2} \oplus A_1$, by determining $H_X$ and the action on $\Pic(X)$. Moreover, in the case $\mathrm{char}(k)=2$ below, $H_X \cap \Aut_s(X) \neq 1$, hence the $p$-group that Proposition \ref{trivialautoprop} predicts is realized.

\begin{proposition}\label{exampleauto}
    Let $X$ be a K3 surface over $k$ with $\Pic(X)\cong U\oplus E_8 ^{\oplus 2} \oplus A_1$. Then $\Aut(X)\cong S_3 \times \mathbb{Z}/2\mathbb{Z}$.
\end{proposition}

\begin{proof}   
    The dual graph of rational curves on $X$ was determined by Kond\={o} in \cite{kondofiniteauto}:   
\begin{center}
\begin{tikzpicture}[scale=0.9]
\draw (-3,0) -- (3,0);
\draw (0,5) -- (3,0);
\draw (0,5) -- (-3,0);
\draw (-3/2,5/2) -- (3,0);
\draw (3/2,5/2) -- (-3,0);
\draw (0,0) -- (0,5);

\draw (0,5) node {$\bullet$};
\draw (1/2,-5/3*1/2 + 5) node {$\bullet$};
\draw (1, -5/3*2/2 + 5) node {$\bullet$};
\draw (3/2,-5/3*3/2 + 5) node {$\bullet$};
\draw (2,-5/3*4/2 + 5) node {$\bullet$};
\draw (5/2, -5/3*5/2 + 5) node {$\bullet$};
\draw (3, -5/3*6/2 + 5) node {$\bullet$};

\draw (-1/2,5/3*-1/2 + 5) node {$\bullet$};
\draw (-1, 5/3*-2/2 + 5) node {$\bullet$};
\draw (-3/2,5/3*-3/2 + 5) node {$\bullet$};
\draw (-2,5/3*-4/2 + 5) node {$\bullet$};
\draw (-5/2, 5/3*-5/2 + 5) node {$\bullet$};
\draw (-3, 5/3*-6/2 + 5) node {$\bullet$};

\draw (0.75,5/9*0.75+5/3) node {$\bullet$};
\draw (-0.75,5/9*0.75+5/3) node {$\bullet$};
\draw (0, 5/6) node {$\bullet$};
\draw (0.75,5/9*0.75+5/3) node [above]{$A_{22}$};
\draw (-0.75,5/9*0.75+5/3) node [above]{$A_{23}$};
\draw (0, 5/6) node  [right]{$A_{24}$};

\draw [dotted][very thick] (0, 5/6) -- (-0.75,5/9*0.75+5/3);
\draw [dotted][very thick] (0, 5/6) -- (0.75,5/9*0.75+5/3);
\draw [dotted][very thick] (0.75,5/9*0.75+5/3) -- (-0.75,5/9*0.75+5/3);

\draw (-1.5, 5/9*-1.5+5/3) node {$\bullet$};
\draw (1.5, 5/9*-1.5+5/3) node {$\bullet$};
\draw (0, -5/3*2/2 + 5) node {$\bullet$};
\draw [very thick] (0, -5/3*2/2 + 5) -- (0,0);
\draw [very thick] (-1.5, 5/9*-1.5+5/3) -- (1.5,-5/3*3/2 + 5);
\draw [very thick] (1.5, 5/9*-1.5+5/3) -- (-1.5,-5/3*3/2 + 5);
\draw (-1.5, 5/9*-1.5+5/3) node [below]{$A_{19}$};
\draw (1.5, 5/9*-1.5+5/3) node [below]{$A_{20}$};
\draw (0, -5/3*2/2 + 5) node [right]{$A_{21}$};

\draw (-3,0) node {$\bullet$};
\draw (-2,0) node {$\bullet$};
\draw (-1,0) node {$\bullet$};
\draw (0,0) node {$\bullet$};
\draw (1,0) node {$\bullet$};
\draw (2,0) node {$\bullet$};
\draw (3,0) node {$\bullet$};

\draw (-3,0) node [below]{$A_{1}$};
\draw (-2,0) node [below]{$A_{2}$};
\draw (-1,0) node [below]{$A_{3}$};
\draw (0,0) node [below]{$A_{4}$};
\draw (1,0) node [below]{$A_{5}$};
\draw (2,0) node [below]{$A_{6}$};
\draw (3,0) node [below]{$A_{7}$};

\draw (1/2,-5/3*1/2 + 5) node [right]{$A_{12}$};
\draw (1, -5/3*2/2 + 5) node  [right]{$A_{11}$};
\draw (3/2,-5/3*3/2 + 5) node  [right]{$A_{10}$};
\draw (2,-5/3*4/2 + 5) node  [right]{$A_{9}$};
\draw (5/2, -5/3*5/2 + 5) node  [right]{$A_{8}$};
\draw (0,5) node [above]{$A_{13}$};
\draw (-1/2,5/3*-1/2 + 5) node [left]{$A_{14}$};
\draw (-1, 5/3*-2/2 + 5) node  [left]{$A_{15}$};
\draw (-3/2,5/3*-3/2 + 5) node  [left]{$A_{16}$};
\draw (-2,5/3*-4/2 + 5) node  [left]{$A_{17}$};
\draw (-5/2, 5/3*-5/2 + 5) node  [left]{$A_{18}$};

\end{tikzpicture}
\end{center} 
    Here, a thick line denotes an intersection number equal to $2$ and a dotted line denotes an intersection number equal to $6$. Denote this graph by $\Gamma$ and observe that $\Aut(\Gamma)=S_3$. There is a natural representation $\rho \colon \Aut(X) \to \Aut(\Gamma)$ and we claim that $\rho$ is surjective and admits a section.
    
    The pencil $\vert A_1+\dots +A_{18}\vert$ induces an elliptic fibration $\pi \colon X \to \mathbb{P}^1$ with one reducible singular fiber of type $I_{18}$ and sections $A_{19}$, $A_{20}$ and $A_{21}$. Consider the group $\Aut_{\pi}(X/ \mathbb{P}^1)$ of $\mathbb{P}^1$-automorphisms of $X$, that is, the group of automorphisms of $X$ that preserve $\pi$ and act trivially on the base. We will show that $\Aut_{\pi}(X/\mathbb{P}^1)=S_3$ and that it realizes every automorphism of $\Gamma$. To understand the action on the graph it suffices to understand the action on the sections of the fibration, which are in bijection with the set of $k(t)$-rational points of the generic fiber $E_{\pi}$. Restriction and spreading out yield an isomorphism $\Aut_{\pi}(X/\mathbb{P}^1) \cong \Aut(E_{\pi})$, hence we can compute everything on the generic fiber. In the following, we identify the $k(t)$-rational points of $E_{\pi}$ with the sections $A_{19}$, $A_{20}$, $A_{21}$ and we choose $A_{19}$ as the origin $O$ of $E_{\pi}$. Now, $\Aut(E_{\pi})= T\rtimes \Aut(E_{\pi},O)$, where $T$ is the normal subgroup generated by translations along $k(t)$-rational points and $\Aut(E_{\pi},O)$ is the group of elliptic curve automorphisms with $O$ as the origin. We directly see that $T\cong E_\pi(k(t))\cong \mathbb{Z}/3\mathbb{Z}$. Also $\Aut(E_{\pi},O)\cong \mathbb{Z}/2\mathbb{Z}$, necessarily generated by the inversion homomorphism, because the presence of a fiber of type $I_{18}$ implies that $j(E_{\pi})$ has a pole and \cite[Appendix A, Prop. 1.2.(c)]{Silverman2009}. Translation by the $3$-torsion section $A_{20}$ permutes the three sections and so acts by rotation on $\Gamma$. The involution which is induced by the inversion homomorphism on $E_{\pi}$ acts by reflection on $\Gamma$, because it permutes $A_{20}$ and $A_{21}$ while preserving $A_{19}$. Hence, the natural map $\Aut(X) \to \Aut(\Gamma)$ is surjective and the following short exact sequence splits: \[1 \to H_X \to \Aut(X) \to \Aut(\Gamma) \to 1.\]
    
    It remains to compute $H_X$. The pencil $\vert A_{24}+A_4 \vert$ induces an elliptic fibration $\pi' \colon X \to \mathbb{P}^1$ with two reducible singular fibers of type $I_2$/ $III$ and $I_{12}^*$ and sections $A_3$ and $A_5$. The automorphism $\sigma$ of $X$ that is induced by the inversion homomorphism of the generic fiber $E_{\pi'}$ is an involution that acts trivially on $\Gamma$ (because it preserves the sections $A_3$ and $A_5$), which means that $\langle\sigma \rangle\subseteq H_X$. To show equality, we determine the order of $H_X$. For this we use a third elliptic fibration $\pi'' \colon X \to \mathbb{P}^1$ which is induced by the pencil $\vert A_{21}+A_{24} \vert$ and has three reducible singular fibers, two of type $II^*$ and one of type $I_2$/$III$ and one section $A_{13}$. Observe that if $g\in H_X$ then $g$ preserves the three reducible fibers of $\pi''$, so $H_X \subseteq \Aut_{\pi''}(X/\mathbb{P}^1)$. Since $\pi''$ admits the unique section $A_{13}$ the same arguments as in the previous paragraph show that $\Aut_{\pi''}(X/\mathbb{P}^1)\cong \Aut( E_{\pi''},A_{13})$. If we show that $j(E_{\pi''}) \neq0$, $1728$, then the latter group is $\mathbb{Z}/2\mathbb{Z}$ by \cite[Appendix A, Prop. 1.2.(c)]{Silverman2009}. For this, we give an explicit Weierstrass model. Without loss of generality we can assume that the singular fibers are over $0$, $1$ and $\infty$. Then, one can run Tate's algorithm to check that a Weierstrass model is given affinely by an equation of the form \[y^2+a_1txy=x^3+a_2t^2x^2+t^5(t-1)^2\]
    with $a_1,a_2 \in k$, $a_1^2+4a_2\neq 0$. A straightforward calculation shows that $j(E_{\pi''}) \neq0$, $1728$.
\end{proof}

\begin{remark}
    As pointed out to us by Kond\=o, the following interesting phenomenon appears in the next possible case, namely for K3 surfaces with Picard lattice isometric to $U\oplus E_8^{\oplus2}$. In characteristic $2$, a Kummer K3 surface $X$ associated with the product of an ordinary and a supersingular elliptic curve has $\Pic(X) \cong U\oplus E_8^{\oplus2}$ and $\Aut(X) \cong \mathrm{SL}_2(\mathbb{F}_3) \times \mathbb{Z}/2\mathbb{Z}$ with $H_X=\mathrm{SL}_2(\mathbb{F}_3)$ (\cite[Section 6]{kondo2025automorphismgroupskummersurfaces}). However, the group $\mathrm{SL}_2 (\mathbb{F}_3)$ cannot appear as $H_X$ for any K3 surface $X$ with $\Pic(X) \cong U \oplus E_8^{\oplus 2}$ when $\mathrm{char}(k) \neq 2$. Indeed, if $\mathrm{char}(k)=0$, then $H_X \subseteq \mathrm{Aut}(X)/\mathrm{Aut}_s(X)$ and the latter group is cyclic, while if $\mathrm{char}(k)\geq 3$ then $\mathrm{SL}_2(\mathbb{F}_3)$ does not satisfy Proposition \ref{trivialautoprop}.
\end{remark}

\begin{remark}
    It is plausible that an analysis similar to the one in the proof of Proposition \ref{exampleauto} can be used to compute the automorphism groups in the remaining cases. The symmetries of the graph of rational curves can often be realized by using appropriate elliptic fibrations. The main difficulty lies in computing $H_X$, particularly due to the possible presence of wild automorphisms. In general we expect that, for every Picard lattice in Table \ref{table1} and for every positive characteristic, the automorphism group of a general K3 surface with that Picard lattice is isomorphic to that of a general complex K3 surface with the same Picard lattice, as computed by Kond\=o in \cite{kondofiniteauto}.
\end{remark}

\resumetoc

\newpage

\appendix

\section{Classification table and discriminant data}\label{table}

In this appendix, we present the two tables referenced in this paper. In Table \ref{table1}, we list the Picard lattices of complex K3 surfaces with finite automorphism group. In Table \ref{table2}, we list the lattices of $\mathrm{rank}\geq 10$ of Table \ref{table1} together with the calculations of the minimal number of generators of their discriminant groups. Table \ref{table2} is intended to supplement the proof of Proposition \ref{simplification}.

Let us first introduce the lattices $S_i$, $S_{a,b,c}$, $S'_{4,1,2}$, $L(24)$ and $L(27)$.
\begin{enumerate}
\item The lattice $S_{1,1,1}$ is the rank $3$ lattice generated by vectors $e_1$, $e_2$, $e_3$ with Gram matrix 
\[\begin{pmatrix}
    -2 & 0 & 1 \\
    0 & -2 & 2 \\
    1 & 2 & -2 
\end{pmatrix}.\]
\item For $a$, $b$, $c \in \mathbb{Z}$, the lattice $S_{a,b,c}$ is the sublattice of $S_{1,1,1}$ generated by $ae_1$, $be_2$, $ce_3$. 
\item The lattice $S'_{4,1,2}$ is the sublattice of $S_{1,1,1}$ generated by $2e_1 + e_3$, $e_2$, $2e_3$.
\item The lattice $S_0$ is the rank $2$ lattice with Gram matrix 

\[\begin{pmatrix}
    0& -3\\
    -3& -2
\end{pmatrix}.\]
\item $S_1=\mathbb{Z}(6)\oplus A_1^{\oplus 2}$.
\item $S_2=\mathbb{Z}(36)\oplus A_2$.
\item $S_3= \mathbb{Z}(12)\oplus A_2$. 
\item The lattice $S_4$ is the rank $3$ lattice with Gram matrix 
\[\begin{pmatrix}
    2& 1 & 2 \\
    1 & -2 &1\\
    2 & 1 & -2
\end{pmatrix}.\]
\item $S_5=\mathbb{Z}(4)\oplus A_2$.
\item The lattice $S_6$ is the rank $3$ lattice with Gram matrix
\[\begin{pmatrix}
    2&2&3\\
    2&-2&1\\
    3&1&-2
\end{pmatrix}.\]
\item The lattices $L(24)$ and $L(27)$ are the rank $4$ lattices with Gram matrices, respectively
\[\begin{pmatrix}
    2&1&1&1\\
    1&-2&0&0\\
    1&0&-2&0\\
    1&0&0&-2
\end{pmatrix},
\begin{pmatrix}
    12&2&0&0\\
    2&-2&1&0\\
    0&1&-2&1\\
    0&0&1&-2
\end{pmatrix}.\]
\end{enumerate}

\begin{longtable}{|c|c|c|c|c|}
\captionsetup{font=small}
     \caption{Picard lattices of complex K3 surfaces with finite automorphism group}\label{table1} \\
     
     \hline
     \multicolumn{5}{|c|}{\textbf{Rank 1}}\\
     \hline
     \multicolumn{5}{|c|}{$\mathbb{Z}(2d)$, $d>0$}\\
     \hline
     \multicolumn{5}{|c|}{\textbf{Rank 2}}\\
     \hline
     \multicolumn{5}{|c|}{Even hyperbolic lattices such that there exists $\alpha\neq0$ with $\alpha^2=0$ or $\alpha^2=-2$}\\
     \hline
     \multicolumn{5}{|c|}{\textbf{Rank 3} }\\
     \hline
     $S_1$& $S_2$ & $S_3$&$S_4$ &$S_5$\\
     \hline
     $S_6$& $S_{1,1,1}$&$S_{1,1,2}$&$S_{1,1,3}$&$S_{1,1,4}$\\
     \hline
     $S_{1,1,6}$&$S_{1,1,8}$&$S_{1,2,1}$&$S_{1,3,1}$ &$S_{1,4,1}$\\
     \hline
     $S_{1,5,1}$ &$S_{1,6,1}$&$S_{1,9,1}$&$S_{4,1,1}$ &$S_{5,1,1}$\\
     \hline
     $S_{6,1,1}$&$S_{7,1,1}$ &$S_{8,1,1}$ &$S_{10,1,1}$&$S_{12,1,1}$\\
     \hline
     $S'_{4,1,2}$ & \multicolumn{4}{c|}{  }\\
     \hline
     \multicolumn{5}{|c|}{\textbf{Rank 4} } \\
     \hline
     $L(24)$&$L(27)$&$\mathbb{Z}(8)\oplus A_1 ^{\oplus 3}$ &$U\oplus A_1 ^{\oplus 2}$ & $U(2)\oplus A_1 ^{\oplus 2}$\\
     \hline
     $U(3)\oplus A_1 ^{\oplus 2}$&$U(4)\oplus A_1 ^{\oplus 2}$& $U\oplus A_2$&$U(2)\oplus A_2$ &$U(3)\oplus A_2$\\  
     \hline
     $U(6)\oplus A_2$&$S_0 \oplus A_2$&$\mathbb{Z}(4)\oplus \mathbb{Z}(-4) \oplus A_2$& $\mathbb{Z}(4)\oplus A_3$&  \\
     \hline
     \multicolumn{5}{|c|}{\textbf{Rank 5}} \\
     \hline
     $U\oplus A_1 ^{\oplus 3}$&$U(2)\oplus A_1 ^{\oplus 3}$ &$U(4)\oplus A_1 ^{\oplus 3}$ &$U\oplus A_1 \oplus A_2$ &$U\oplus A_3$\\
     \hline
     $\mathbb{Z}(4) \oplus D_4$& $\mathbb{Z}(8) \oplus D_4$ & $\mathbb{Z}(16) \oplus D_4$ & $\mathbb{Z}(6)\oplus A_2 ^{\oplus2}$ &  \\
     \hline
     \multicolumn{5}{|c|}{\textbf{Rank 6}}\\
     \hline
     $U(3)\oplus A_2 ^{\oplus2}$&$U(4)\oplus D_4$ &$U\oplus A_4$ &$U\oplus A_1 \oplus A_3$ & $U\oplus A_2 ^{\oplus 2}$\\
     \hline
     $U\oplus A_1 ^{\oplus 2}\oplus A_2$&$U(2)\oplus A_1^{\oplus4}$ & $U\oplus A_1^{\oplus4}$ &$U(2)\oplus D_4$ &$U\oplus D_4$  \\
     \hline
     \multicolumn{5}{|c|}{\textbf{Rank 7}}\\
     \hline
     $U\oplus D_4\oplus A_1$&$U\oplus A_1^{\oplus5}$ & $U(2)\oplus A_1^{\oplus 5}$&$U\oplus A_1\oplus A_2^{\oplus 2}$ & $U\oplus A_1^{\oplus 2} \oplus A_3$ \\
     \hline
     $U\oplus A_2\oplus A_3$& $U\oplus A_1\oplus A_4$& $U\oplus A_5$&$U\oplus D_5$ &   \\
     \hline 
     \multicolumn{5}{|c|}{\textbf{Rank 8}}\\
     \hline
     $U\oplus D_6$&$U\oplus D_4\oplus A_1^{\oplus2}$ &$U\oplus A_1^{\oplus 6}$ &$U(2)\oplus A_1^{\oplus 6}$ & $U\oplus A_2^{\oplus 3}$\\
     \hline
     $U\oplus A_3^{\oplus 2}$& $U\oplus A_2\oplus A_4$&$U\oplus A_1\oplus A_5$ &$U\oplus A_6$ &$U\oplus A_2\oplus D_4$ \\
     \hline 
     $U\oplus A_1\oplus D_5$&$U\oplus E_6$ & \multicolumn{3}{c|}{   }\\
     \hline
     \pagebreak
     \hline
     \multicolumn{5}{|c|}{\textbf{Rank 9}}\\
     \hline
     $U\oplus E_7$&$U\oplus D_6\oplus A_1$ &$U\oplus D_4 \oplus A_1^{\oplus 3}$ & $U\oplus A_1^{\oplus 7}$ & $U(2)\oplus A_1^{\oplus 7}$\\
     \hline
     $U\oplus A_7$&$U\oplus D_4\oplus A_3$& $U\oplus D_5 \oplus A_2$& $U\oplus D_7$& $U\oplus A_1 \oplus E_6$\\
     \hline
     \multicolumn{5}{|c|}{\textbf{Rank 10}}\\
     \hline 
     $U\oplus E_8$& $U\oplus D_8$ &$U\oplus E_7 \oplus A_1$& $U\oplus D_4^{\oplus 2}$ & $U\oplus D_6\oplus A_1^{\oplus 2}$\\
     \hline
     $U(2)\oplus D_4^{\oplus 2}$& $U\oplus D_4\oplus A_1^{\oplus 4}$ &$U\oplus A_1^{\oplus 8}$&$U\oplus A_2 \oplus E_6$& \\
     \hline
     \multicolumn{5}{|c|}{\textbf{Rank 11}}\\
     \hline 
     $U\oplus E_8\oplus A_1$ & $U\oplus D_8\oplus A_1$&$U\oplus D_4^{\oplus 2}\oplus A_1$ &$U\oplus D_4\oplus A_1^{\oplus 5}$& \\
     \hline
     \multicolumn{5}{|c|}{\textbf{Rank 12}}\\
     \hline
     $U\oplus E_8\oplus A_1^{\oplus 2}$ & $U\oplus D_8\oplus A_1^{\oplus 2}$ &$U\oplus D_4^{\oplus 2}\oplus A_1^{\oplus 2}$ & $U\oplus A_2 \oplus E_8$& \\
     \hline
     \multicolumn{5}{|c|}{\textbf{Rank 13}}\\
     \hline
     $U\oplus E_8\oplus A_1^{\oplus 3}$ & $U\oplus D_8\oplus A_1^{\oplus 3}$& $U\oplus E_8 \oplus A_3$& \multicolumn{2}{c|}{  }\\
     \hline
     \multicolumn{5}{|c|}{\textbf{Rank 14}}\\
     \hline
     $U\oplus E_8 \oplus D_4$& $U\oplus D_8\oplus D_4$&$U\oplus E_8 \oplus A_1^{\oplus 4}$& \multicolumn{2}{c|}{  }\\
     \hline
     \multicolumn{5}{|c|}{\textbf{Rank $\geq$ 15}}\\
     \hline
     $U\oplus E_8 \oplus D_4\oplus A_1$& $U\oplus E_8 \oplus D_6$& $U\oplus E_8 \oplus E_7$ & $U\oplus E_8^{\oplus 2}$& $U\oplus E_8^{\oplus 2} \oplus A_1$\\
     \hline
     
\end{longtable}
\clearpage

\begin{longtable}{|c|c|c|c|}
\captionsetup{font=small}
\caption{Supplement to the proof of Proposition \ref{simplification}}\label{table2}\\
    \hline
    Lattice $\Lambda$&$\mathrm{Rank}\, \Lambda $&$l(A_{\Lambda})$&$18-\mathrm{rank}\, \Lambda$\\
    \hline
    $U\oplus E_8$ & 10 &0 &8 \\
    \hline
    $U\oplus D_8$ & 10 &2 & 8\\
    \hline
    $U\oplus E_7 \oplus A_1$ & 10 &2 &8 \\
    \hline
    $U\oplus D_4^{\oplus 2}$ & 10 & 4& 8\\
    \hline
    $U\oplus D_6\oplus A_1^{\oplus 2}$ & 10 &4 &8 \\
    \hline
    $U(2)\oplus D_4^{\oplus 2}$ & 10 &6& 8\\
    \hline
    $U\oplus D_4\oplus A_1^{\oplus 4}$ & 10 &6 &8 \\
    \hline
    $U\oplus A_1^{\oplus 8}$ & 10 & 8& 8\\
    \hline
    $U\oplus A_2 \oplus E_6$ & 10 & 2& 8\\
    \hline
    $U\oplus E_8\oplus A_1$ & 11 & 1& 7\\
    \hline
    $U\oplus D_8\oplus A_1$ & 11 & 3& 7\\
    \hline
    $U\oplus D_4^{\oplus 2}\oplus A_1$ & 11 &5 &7 \\
    \hline
    $U\oplus D_4\oplus A_1^{\oplus 5}$ & 11 & 7& 7\\
    \hline
    $U\oplus E_8\oplus A_1^{\oplus 2}$ & 12 & 2& 6\\
    \hline
    $U\oplus D_8\oplus A_1^{\oplus 2}$ & 12 & 4& 6\\
    \hline
    $U\oplus D_4^{\oplus 2}\oplus A_1^{\oplus 2}$ & 12 &6 &6 \\
    \hline
    $U\oplus A_2 \oplus E_8$ & 12 & 1& 6\\
    \hline
    $U\oplus E_8\oplus A_1^{\oplus 3}$ & 13 &3 &5 \\
    \hline
    $U\oplus D_8\oplus A_1^{\oplus 3}$ & 13 & 5& 5\\
    \hline
    $U\oplus E_8 \oplus A_3$ & 13 & 1& 5\\
    \hline
    $U\oplus E_8 \oplus D_4$ & 14 & 2& 4\\
    \hline
    $U\oplus D_8\oplus D_4$ & 14 & 4& 4\\
    \hline
    $U\oplus E_8 \oplus A_1^{\oplus 4}$ & 14 &4 &4 \\
    \hline
    $U\oplus E_8 \oplus D_4\oplus A_1$ &  15 & 3& 3\\
    \hline
    $U\oplus E_8 \oplus D_6$ & 16 & 2& 2\\
    \hline
    $U\oplus E_8 \oplus E_7$ & 17 & 1& 1\\
    \hline
\end{longtable}

\printbibliography

\end{document}